\documentclass[11pt]{amsart}
\usepackage{setspace, amsmath, amsthm, amssymb, amsfonts, amscd, epic, graphicx, ulem, dsfont}
\usepackage[T1]{fontenc}
\usepackage{multirow}
\usepackage{bbm}
\usepackage{enumerate}

\makeatletter \@namedef{subjclassname@2010}{
  \textup{2010} Mathematics Subject Classification}
\makeatother

\newtheorem{thm}{Theorem}[section]
\newtheorem{cor}[thm]{Corollary}
\newtheorem{lem}[thm]{Lemma}
\newtheorem{pro}[thm]{Proposition}

\theoremstyle{remark}
\newtheorem*{rema}{Remark}

\theoremstyle{definition}

\newcommand{\REAL}{\text{\rm{Re}}}
\newcommand{\Ima}{\text{\rm{Im}}}

\newcommand{\Z}{\mathbb{Z}}

\newcommand{\C}{\mathbb{C}}

\begin{document}

\title[Absolute Value of The Product and the Sum of Operators]{On The Absolute Value of The Product and the Sum of Linear Operators}
\author[M. H. Mortad]{Mohammed Hichem Mortad}

\dedicatory{}
\thanks{}
\date{}
\keywords{Absolute Value. Triangle Inequality. Normal, Hyponormal,
Self-adjoint and Positive Operators. Commutativity. Fuglede Theorem.
}

\subjclass[2010]{Primary 47A63, Secondary 47A62, 47B15, 47B20.}

\address{Department of
Mathematics, University of Oran 1, Ahmed Ben Bella, B.P. 1524, El
Menouar, Oran 31000, Algeria.\newline {\bf Mailing address}:
\newline Pr Mohammed Hichem Mortad \newline BP 7085 Seddikia Oran
\newline 31013 \newline Algeria}

\email{mhmortad@gmail.com, mortad@univ-oran.dz.}

\begin{abstract}Let $A,B\in B(H)$. In the present paper, we establish simple and interesting facts on when we
have $|A||B|=|B||A|$, $|AB|=|A||B|$, $|A\pm B|\leq |A|+|B|$,
$||A|-|B||\leq |A\pm B|$ and $\||A|-|B|\|\leq \|A\pm B\|$, where
$|\cdot|$ denotes the absolute value (or modulus) of an operator.
The results give some other interesting consequences.
\end{abstract}

\maketitle

\section{Introduction}

Let $H$ be a complex Hilbert space and let $A,B\in B(H)$. We say
that $A$ is positive, and we write $A\geq0$, if $<Ax,x>\geq 0$ for
all $x\in H$. Since $H$ is a complex Hilbert space, a positive
operator is clearly self-adjoint. We say that $A\geq B$ if they are
both self-adjoint and $A-B\geq 0$. Recall also that if $A\geq 0$,
then there is a unique positive operator $B$ such that $B^2=A$. We
call it the (positive) square root of $A$ and we denote it by
$\sqrt{A}$ (or $A^{\frac{1}{2}}$). Next, we gather basic results on
square roots of sums and products.

\begin{lem}\label{sqrt A+B Lemma}
Let $A,B\in B(H)$ be such that $AB=BA$ and $A,B\geq 0$. Then
\begin{itemize}
  \item $AB\geq 0$.
  \item $\sqrt{AB}=\sqrt A\sqrt B$.
  \item $\sqrt{A+B}\leq \sqrt A+\sqrt B.$
\end{itemize}
\end{lem}

The unique positive square root of the positive operator $A^*A$ is
commonly known as the absolute value (or modulus) of $A$. We denote
it by $|A|$, that is, $|A|=\sqrt{A^*A}$. Notice that
$\|A\|=\|~|A|~\|$ always holds.

We usually warn students to be careful with this notation as it may
mislead them to think that e.g.
\[|A|=|A^*|,~|A+B|\leq |A|+|B| \text{ or } |AB|=|A||B|\]
would hold. Counterexamples are easily found in the setting of 2 by
2 matrices. Notice that $A$ is normal if and only if $|A|=|A^*|$.

Also, a priori if $A,B$ are arbitrary, then there is no reason why
we should expect $|A||B|=|B||A|$ to hold (for instance, just think
of positive operators). Even when $AB=BA$, the equality
$|A||B|=|B||A|$ need not hold. For example, let
\[A=\left(
      \begin{array}{cc}
        1 & 1 \\
        0 & 1 \\
      \end{array}
    \right)\text{ and } B=\left(
                            \begin{array}{cc}
                              0 & 1 \\
                              0 & 0 \\
                            \end{array}
                          \right).
\]
Then as we can easily verify:
\[AB=\left(
       \begin{array}{cc}
         0 & 1 \\
         0 & 0 \\
       \end{array}
     \right)=BA
\]
whereas $|A||B|\neq |B||A|$.

Observe that we have purposely avoided normal operators in our
counterexample (cf. Proposition \ref{commutativity operators equiv
abs va pro}).

The main aim of this paper is to investigate when relations of the
types
\begin{itemize}
  \item $|A||B|=|B||A|$;
  \item $|AB|=|A||B|$;
  \item $|A+B|\leq |A|+|B|$;
  \item $|~|A|-|B|~|\leq |A+B|$;
  \item $\|~|A|-|B|~\|\leq \|A+B\|$;
\end{itemize}
hold. It turns out that normality and sometimes hyponormality plus
commutativity are sufficient for these relations to hold. This comes
to corroborate the resemblance to complex numbers which is already
known to many. Notice also that commutativity is not unnatural as we
already have it in $(\C,\times)$.

The idea here is to start from scratch, and use as basic results as
possible to make the paper accessible to a wide
 audience. We note that for example, we have wittingly avoided the use of the spectral theorem of normal
 operators. Therefore, most of the results here can be taught at
 elementary courses in Operator Theory.

It is worth noticing that there is a big amount of papers which have
dealt with inequalities involving absolute values and/or norms of
operators. The literature is so rich that we rather refer readers to
books which have gathered most of these results. For example, see
\cite{Bhatia-book-matrices}, \cite{FUR.book} and
\cite{Zhan-inequalities}.

Finally, we assume the reader is familiar with other basic results
on Operator Theory. A well established reference is \cite{Con}. We
do recall two crucial results though.

\begin{thm}(\textbf{L\"{o}wner-Heinz Inequality}, see
\cite{Pdersen-1972} for a simple proof)\label{Lowner-Heinz
inequality theorem} If $A\geq B\geq 0$, then $A^{\alpha}\geq
B^{\alpha}$ for any $\alpha\in [0,1]$.
\end{thm}

\begin{rema}
It is known to readers that $A\geq B\geq 0$ implies that $A^2\geq
B^2$ when $AB=BA$.
\end{rema}

Since we will be dealing with sums and products of commuting normal
operators, the use of the celebrated Fuglede-Putnam theorem is
inevitable. \textit{The following lemma will be used below without
further notice.}

\begin{lem}\label{Fugelde equivalences lemma}Let $A,B\in B(H)$ where $A$ is normal. Then, we have
\[AB=BA\Longleftrightarrow A^*B=BA^*\Longleftrightarrow AB^*=B^*A\Longleftrightarrow A^*B^*=B^*A^*.\]
\end{lem}

\section{Main Results: Absolute Value and Products}

We start with the following:

\begin{pro}\label{commutativity operators equiv abs va pro}
Let $A,B\in B(H)$ be such that $AB=BA$. If $A$ is normal, then
$|A||B|=|B||A|$.
\end{pro}

\begin{rema}
The preceding result was proved in \cite{Zeng Young inequality
paper} by assuming that \textit{both }$A$ and $B$ are normal.
\end{rema}

\begin{proof}Since $AB=BA$ and $A$ is normal, we have $A^*B=BA^*$. We then clearly have from the previous two relations:
\[AB=BA\Longrightarrow A^*AB=A^*BA=BA^*A.\]
Hence
\[|A|B=B|A|.\]
Since $|A|$ is self-adjoint, the previous equality gives (by taking
adjoints) $|A|B^*=B^*|A|$. Hence
\[B^*|A|B=|A|B^*B\Longrightarrow B^*B|A|=|A|B^*B\Longrightarrow |B||A|=|A||B|,\]
as required.
\end{proof}

We have already observed above that in general $|AB|\neq|A||B|$. The
following result is somewhat inspired by a one in
\cite{Gustafson-Mortad-2016}.

\begin{thm}\label{ |AB|=|A||B| THM}
Let $A,B\in B(H)$ be self-adjoint such that $AB$ is normal. Then
\[|AB|=|A||B|.\]
\end{thm}

\begin{rema}
It was noted in \cite{Harte} that if $S,T$ are two non-commuting
self-adjoint operators, then the \textit{inequality} $|ST|\leq
|S||T|$ never holds. So, in our result the normality of the product
transforms the non valid inequality into a true full equality.
\end{rema}

\begin{rema}
Notice that $AB$ being a normal product of two self-adjoint
operators does not necessarily imply that $AB$ is self-adjoint, i.e.
we do not necessarily have $AB=BA$. If, however, we impose further
that $A\geq 0$ (or $B\geq 0$), then $AB$ becomes self-adjoint. See
e.g. \cite{Mortad-PAMS2003}.
\end{rema}

\begin{proof}Since $A$ and $B$ are self-adjoint, we may write
\[B(AB)=BAB=(AB)^*B.\]
Since $AB$ and $(AB)^*$ are normal, the Fuglede-Putnam theorem gives
\[B(AB)^*=(AB)^{**}B \text{ or merely } B^2A=AB^2.\]
Consequently, $B^2A^2=AB^2A=A^2B^2$.

On the other hand, we easily see that

\[|AB|^2=(AB)^*AB=AB(AB)^*=AB^2A=A^2B^2\]

and so
\[|AB|=\sqrt{A^2B^2}=\sqrt{A^2}\sqrt{B^2}=|A||B|,\]
as required.
\end{proof}

Since $|AB|$ is self-adjoint, we have:

\begin{cor}Let $A,B\in B(H)$ be self-adjoint such that $AB$ is normal.
Then $|A||B|$ is self-adjoint, i.e. $|A||B|=|B||A|$. Moreover, if we
also assume that $A,B\geq0$, then $AB\geq0$.
\end{cor}

The assumptions of the previous theorem cannot just be dropped. We
give a counterexample for each hypothesis.
\begin{itemize}
  \item Let
  \[A=\left(
        \begin{array}{cc}
          2 & 0 \\
          0 & -1 \\
        \end{array}
      \right) \text{ and } B=\left(
                               \begin{array}{cc}
                                 0 & 1 \\
                                 1 & 0 \\
                               \end{array}
                             \right).
\]
Then each of $A$ and $B$ is self-adjoint but $AB$ is not normal for
\[AB=\left(
       \begin{array}{cc}
         0 & 2 \\
         -1 & 0 \\
       \end{array}
     \right)
\]
We can easily check that
\[|AB|=\left(
         \begin{array}{cc}
           1 & 0 \\
           0 & 2 \\
         \end{array}
       \right),~|A|=\left(
         \begin{array}{cc}
           2 & 0 \\
           0 & 1 \\
         \end{array}
       \right) \text{ and } |B|=\left(
                                  \begin{array}{cc}
                                    1 & 0 \\
                                    0 & 1 \\
                                  \end{array}
                                \right),
\]
i.e.
\[|AB|\neq|A||B|.\]
  \item Let
\[A=\left(
      \begin{array}{cc}
        0 & 1 \\
        2 & 0 \\
      \end{array}
    \right) \text{ and } B=\left(
                             \begin{array}{cc}
                               0 & 2 \\
                               1 & 0 \\
                             \end{array}
                           \right).
\]
Then, neither $A$ nor $B$ is normal. Their product $AB$ is, however,
self-adjoint (hence normal!) because
\[AB=\left(
       \begin{array}{cc}
         1 & 0 \\
         0 & 4 \\
       \end{array}
     \right)
\]
Next, we have
\[|A|=\left(
        \begin{array}{cc}
          2 & 0 \\
          0 & 1 \\
        \end{array}
      \right),~|B|=\left(
        \begin{array}{cc}
          1 & 0 \\
          0 & 2 \\
        \end{array}
      \right) \text{ and } |AB|=\left(
        \begin{array}{cc}
          1 & 0 \\
          0 & 2 \\
        \end{array}
      \right).
\]
Accordingly,
\[|AB|=\left(
        \begin{array}{cc}
          1 & 0 \\
          0 & 2 \\
        \end{array}
      \right)\neq \left(
        \begin{array}{cc}
          2 & 0 \\
          0 & 2 \\
        \end{array}
      \right)=|A||B|.\]
\end{itemize}

An akin result to Theorem \ref{ |AB|=|A||B| THM} is:

\begin{thm}\label{abs value A, B normal comm AB=A B THM}
Let $A,B\in B(H)$ be such that $AB=BA$. If $A$ is normal, then
\[|AB|=|A||B|.\]
\end{thm}

\begin{rema}
It was also noted in \cite{Harte} that if $S,T$ are \textit{two}
commuting normal operators, then the \textit{inequality} $|ST|\leq
|S||T|$ holds. So, our result here is stronger.
\end{rema}

\begin{proof}Since $AB=BA$ and $A$ is normal, we get
$A^*B=BA^*$ or $AB^*=B^*A$. Hence
\[|AB|^2=(AB)^*AB=B^*A^*AB=A^*B^*AB=A^*AB^*B.\]
By Proposition \ref{commutativity operators equiv abs va pro},
$A^*AB^*B=B^*BA^*A$. Consequently,
\[|AB|=\sqrt{A^*AB^*B}=\sqrt{A^*A}\sqrt{B^*B}=|A||B|,\]
as required.
\end{proof}

Before generalizing the previous result, we give some direct
consequences. The first one is a funny application.

\begin{cor}
Let $A,B\in B(H)$ be such that $AB=BA$. If $A$ and $B$ are normal,
then
\[|AB|=|A^*B|=|AB^*|=|A^*B^*|=|B^*A^*|=|B^*A|=|BA^*|=|BA|.\]
\end{cor}

\begin{proof}
Since $A$ and $B$ are normal, $|A|=|A^*|$ and $|B|=|B^*|$. As
$AB=BA$, then $|A||B|=|B||A|$ for $A$ (or $B$!) is normal. Now,
apply Theorem \ref{abs value A, B normal comm AB=A B THM} to each of
the eight products.
\end{proof}

\begin{cor}
Let $A,B\in B(H)$ be such that $AB=BA$. If $A$ is normal and $B$ is
invertible, then
\[|AB^{-1}|=|A||B^{-1}|.\]
\end{cor}

\begin{proof}Since $AB=BA$ and $B$ is
invertible, we have $AB^{-1}=B^{-1}A$. Theorem \ref{abs value A, B
normal comm AB=A B THM} does the remaining job.
\end{proof}

\begin{cor}\label{absolute invertible one normal coroll}
Let $A\in B(H)$ be normal and invertible. Then
\[|A^{-1}|=|A|^{-1}.\]
\end{cor}

\begin{proof}It is clear that
\[I=|AA^{-1}|=|A||A^{-1}|.\]
So, the self-adjoint $|A|$ is right invertible and so it is
invertible (cf. \cite{Dehimi-Mortad-II}) and:
\[|A|^{-1}=|A^{-1}|.\]
\end{proof}

Theorem \ref{abs value A, B normal comm AB=A B THM} may be
generalized as follows:

\begin{pro}\label{abso valu product n normal proposition}
Let $(A_i)_{i=1,\cdots, n}$ be a family of pairwise commuting
elements of $B(H)$. If all $(A_i)_{i=1,\cdots, n}$ \textbf{but one}
are normal, then
\[|A_1A_2\cdots A_{n-1}A_n|=|A_1||A_2|\cdots|A_{n-1}||A_n|.\]
\end{pro}

\begin{proof}If $A_1,A_2,\cdots,A_{n-1}$ are normal, then just apply
the preceding theorem by using a proof by induction. Otherwise, just
use commutativity to push the non normal factor to the right as many
times as possible until it will be the last factor on the right of
the product "$\prod_{i=1}^nA_i$". Then proceed as just indicated
three lines above.
\end{proof}

It is simple to see that $|A^2|=|A|^2$ does not hold in general. For
instance, let
\[A=\left(
      \begin{array}{cc}
        0 & 2 \\
        1 & 0 \\
      \end{array}
    \right).
\]
Then
\[|A^2|=\left(
          \begin{array}{cc}
            \sqrt 2 & 0 \\
            0 & \sqrt 2 \\
          \end{array}
        \right)\neq |A|^2=\left(
      \begin{array}{cc}
        2 & 0 \\
        0 & 1 \\
      \end{array}
    \right).
\]

But for normal $A$, things are better.

\begin{cor}
Let $A\in B(H)$ be normal and invertible. Let $n\in\Z$. Then
\[|A^n|=|A|^n.\]
\end{cor}

\begin{proof}The case $n\geq 0$ follows from Proposition \ref{abso valu product n normal
proposition}. The case $n<0$ follows from Proposition \ref{abso valu
product n normal proposition} and Corollary \ref{absolute invertible
one normal coroll}.
\end{proof}

\section{Main Results: Absolute Value and Sums}

We now turn to the triangle inequality w.r.t. $|\cdot|$. We have two
different versions.

Before all else, we state a result (perhaps known to many) which
will be called on later. Its proof relies on the following yet
simpler result.

\begin{lem}\label{anti-symmetric square negative lemma}
If $A\in B(H)$ is anti-symmetric, i.e. $A^*=-A$, then $A^2\leq 0$.
\end{lem}

\begin{lem}\label{Real part smaller than Absolute value Lemma}
Let $T\in B(H)$ be hyponormal (i.e. $TT^*\leq T^*T$, that is,
$\|T^*x\|\leq\|Tx\|$ for all $x\in H$). Then
\[\REAL T=\frac{T+T^*}{2}\leq \sqrt{T^*T}=|T|.\]
\end{lem}

\begin{proof}

It is clear that $T-T^*$ is anti-symmetric and so by Lemma
\ref{anti-symmetric square negative lemma}:\\ $(T-T^*)^2\leq 0$.
Hence
\[(T-T^*)^2\leq 0\Longleftrightarrow T^2+T^{*^{2}}-TT^*-T^*T\leq0.\]

But, $T$ is hyponormal and so $-TT^*-T^*T\geq -2T^*T$. So,
\[T^2+T^{*^{2}}-2T^*T\leq0\]
or
\[T^2+T^{*^{2}}+TT^*+T^*T\leq T^2+T^{*^{2}}+2T^*T\leq 4T^*T.\]
Therefore,
\[(T+T^*)^2\leq 4T^*T \text{ or } |T+T^*|\leq 2|T|\]
 by Theorem \ref{Lowner-Heinz inequality theorem}.

Remembering that $S^{-}=\frac{1}{2}(|S|-S)\geq 0$ whenever $S$ is
self-adjoint, we conclude that
\[T+T^*\leq |T+T^*|\leq 2|T|=2\sqrt{T^*T},\]
as required.
\end{proof}

The following fairly simple result is also useful to us.

\begin{lem}\label{product hyponormal}Let $A,B\in B(H)$ such that $A$ is normal and $B$ is
hyponormal. If $AB=BA$, then $A^*B$ is hyponormal.
\end{lem}

\begin{proof}Let $x\in H$. As $AB=BA$, by the normality of $A$ and the hyponormality of $B$ we have
\[\|(A^*B)^*x\|=\|B^*Ax\|\leq \|BAx\|=\|ABx\|=\|A^*Bx\|,\]
establishing the hyponormality of $A^*B$.
\end{proof}

 Here is the first version of the triangle
inequality.

\begin{thm}\label{triangle inequality abs valu THM} Let $A,B\in B(H)$ be such that $AB=BA$. If $A$ is normal and $B$ is hyponormal,
then the following triangle inequality holds:
\[|A+B|\leq |A|+|B|.\]
\end{thm}

\begin{proof}

Since $A$ is normal and $AB=BA$, we know from Proposition
\ref{commutativity operators equiv abs va pro} that $|A||B|=|B||A|$.
Hence

\begin{align*}|A+B|^2\leq (|A|+|B|)^2&\Longleftrightarrow (A+B)^*(A+B)\leq A^*A+B^*B+2\sqrt{A^*A}\sqrt{B^*B}\\
&\Longleftrightarrow A^*B+B^*A \leq 2\sqrt{A^*A}\sqrt{B^*B}.
\end{align*}

We already know from above that
$\sqrt{A^*A}\sqrt{B^*B}=\sqrt{A^*AB^*B}$. So, to prove the desired
triangle inequality, we are only required to prove
\[A^*B+B^*A \leq 2\sqrt{A^*AB^*B}.\]
But
\[A^*AB^*B=AA^*B^*B=AB^*A^*B=B^*AA^*B.\]

If we set $T=A^*B$, then are done with the proof if we come to show
that the following holds:
\[T+T^*\leq 2\sqrt{T^*T}.\]
But this is just Lemma \ref{Real part smaller than Absolute value
Lemma} once we show that $A^*B$ is hyponormal. This is in effect the
case as $A^*B$ is hyponormal by Lemma \ref{product hyponormal}.

Therefore, under the assumptions of our theorem we have shown that

\[|A+B|^2\leq (|A|+|B|)^2.\]

Hence, by Theorem \ref{Lowner-Heinz inequality theorem}, we have
ended up with
\[|A+B|\leq |A|+|B|,\]
and this is precisely what we wanted to prove.
\end{proof}

\begin{rema}
The foregoing result need not hold if commutativity is dropped even
if $A$ and $B$ are self-adjoint. The reader may check this easily
via the following example:
\[ A=\left(
        \begin{array}{cc}
          -1 & 1 \\
          1 & -1 \\
        \end{array}
      \right)
  \text{ and } B=\left(
                   \begin{array}{cc}
                     2 & 0 \\
                     0 &  0\\
                   \end{array}
                 \right)
  .\]
\end{rema}

\begin{cor}Let $T\in B(H)$ be normal. Then
\[|T|\leq |\REAL T|+|\Ima T|.\]
\end{cor}

\begin{proof}
Write $T=\REAL T+i\Ima T$ where $\REAL T$ and $\Ima T$ are commuting
self-adjoint operators. Then apply Theorem \ref{triangle inequality
abs valu THM}.
\end{proof}

We have another simple consequence of Theorem \ref{triangle
inequality abs valu THM}.
\begin{cor}\label{triangle inequality A-B CORO}
Let $A,B\in B(H)$ be such that $AB=BA$. If $A$ is normal and $B$ is
hyponormal, then the following triangle inequality holds:
\[|A-B|\leq |A|+|B|.\]
\end{cor}

\begin{proof}
Since $AB=BA$, we know that $A(-B)=(-B)A$. Also $-B$ is hyponormal.
Then, apply Theorem \ref{triangle inequality abs valu THM}.
\end{proof}

Theorem \ref{triangle inequality abs valu THM} may be generalized to
a finite sum of operators. Before, recall that the sum of two
commuting normal operators remains normal. This too may be
generalized (the proof by induction is omitted).

\begin{pro}
Let $(A_i)_{i=1,\cdots, n}$ be a family of normal pairwise commuting
elements of $B(H)$. Then $A_1+A_2+\cdots+A_n$ is normal.
\end{pro}

We are ready for the promised generalization of Theorem
\ref{triangle inequality abs valu THM} whose proof is again a proof
by induction.

\begin{cor}
Let $(A_i)_{i=1,\cdots, n}$ be a family of pairwise commuting
elements of $B(H)$. If all $(A_i)_{i=1,\cdots, n}$ are normal
\textbf{except one} which is assumed to be hyponormal, then
\[|A_1+A_2+\cdots+A_n|\leq |A_1|+|A_2|+\cdots+|A_n|.\]
\end{cor}

It is known that an inequality of the type $\||A|-|B|\|\leq \|A\pm
B\|$ is not true in general even if $A$ and $B$ are self-adjoint.

\begin{pro}\label{inequality generalzed proposition}Let $A,B\in B(H)$ be such that $AB=BA$. If $A$ and $B$ are normal, then the following inequality holds:
\[\||A|-|B|\|\leq \|A+B\|.\]
\end{pro}

Probably the following lemma has been noted elsewhere but we state
 it here anyway with a proof.

\begin{lem}\label{-B < A< B norm A leq norm B lemma B>0}
Let $S,T\in B(H)$ be self-adjoint where $S\geq 0$. If $-S\leq T\leq
S$, then $\|T\|\leq \|S\|$.
\end{lem}

\begin{proof}By assumption, for all $x\in H$
\[-<Sx,x>\leq <Tx,x>\leq <Sx,x> \text{ or merely }|<Tx,x>|\leq <Sx,x>.\]
Therefore,
\[\|T\|=\sup_{\|x\|=1}|<Tx,x>|\leq \sup_{\|x\|=1}<Sx,x>=\|S\|,\]
as desired.

\end{proof}

Let us prove Proposition \ref{inequality generalzed proposition}.

\begin{proof}Since $A$ and $B$ are commuting normal operators, we
know that $A+B$ too is normal. Since $A+B$ commutes with $B$, by
Corollary \ref{triangle inequality A-B CORO} we have
\[|A|=|A+B-B|\leq |A+B|+|B|\Longrightarrow |A|-|B|\leq |A+B|.\]
Similarly, as $A+B$ commutes with $A$, we get
\[|B|-|A|\leq |A+B|.\]
Whence
\[-|A+B|\leq|A|-|B|\leq |A+B|.\]
By Lemma \ref{-B < A< B norm A leq norm B lemma B>0} (and
remembering that $\|T\|=\|~|T|~\|$  for $T\in B(H)$), we obtain
\[\||A|-|B|\|\leq \|~|A+B|~\|=\|A+B\|,\]
as required.
\end{proof}

\begin{rema}
In the previous proposition, if $B$ is only hyponormal, then at the
moment we are only sure that:
\[|B|-|A|\leq |A+B|\]
because we can only prove that $A+B$ is hyponormal. We will remedy
this little problem shortly.
\end{rema}

\begin{cor}
Let $A,B\in B(H)$ be such that $AB=BA$. If $A$ and $B$ are normal,
then the following inequality holds:
\[\||A|-|B|\|\leq \|A-B\|.\]
\end{cor}

Proposition \ref{inequality generalzed proposition} can be improved
as it is a particular case of the following remarkable result:

\begin{pro}\label{inequality generalzed proposition -}Let $A,B\in B(H)$ be such that $AB=BA$. If $A$ is normal and $B$ is
hyponormal, then the following inequality holds:
\[||A|-|B||\leq |A-B|.\]
\end{pro}

\begin{proof}
We easily see as $|A||B|=|B||A|$ that
\begin{align*}||A|-|B||^2\leq |A-B|^2&\Longleftrightarrow |A|^2+|B|^2-2|A||B|\leq |A|^2+|B|^2-A^*B-B^*A\\
&\Longleftrightarrow A^*B+B^*A \leq 2\sqrt{A^*A}\sqrt{B^*B}\\
&\Longleftrightarrow A^*B+B^*A \leq 2\sqrt{B^*AA^*B}.
\end{align*}

But, this is always true in virtue of Lemma \ref{Real part smaller
than Absolute value Lemma} as $A^*B$ is hyponormal. Therefore, we
have shown
\[||A|-|B||^2\leq |A-B|^2.\]
A glance at Theorem \ref{Lowner-Heinz inequality theorem} finally
gives
\[||A|-|B||\leq |A-B|.\]
\end{proof}

\begin{cor}\label{inequality generalzed CORO}
Let $A,B\in B(H)$ be such that $AB=BA$. If $A$ is normal and $B$ is
hyponormal, then the following inequality holds:
\[||A|-|B||\leq |A+B|.\]
\end{cor}

\begin{proof}
Since $B$ is hyponormal, so is $-B$. The rest is obvious.
\end{proof}

Here is the improvement of Proposition \ref{inequality generalzed
proposition}:

\begin{cor}
Let $A,B\in B(H)$ be such that $AB=BA$. If $A$ is normal and $B$ is
hyponormal, then the following inequality holds:
\[\||A|-|B|\|\leq \|A\pm B\|.\]
\end{cor}

\begin{proof}By Proposition \ref{inequality generalzed proposition}
and Corollary \ref{inequality generalzed CORO}, we know that
\[||A|-|B||\leq |A\pm B|.\]
Then, calling on Lemma \ref{-B < A< B norm A leq norm B lemma B>0}
yields
\[\||A|-|B|\|=\|~||A|-|B||~\|\leq \|~|A\pm B|~\|=\|A\pm B\|.\]
\end{proof}

If we want to drop commutativity in Theorem \ref{triangle inequality
abs valu THM}, then this is at the cost of adding an extra
condition. Also, we only have to assume that one of the two
operators is normal.

\begin{thm}
Let $A,B\in B(H)$ be such that $AB=BA$. If $A$ is normal and
$A^*B+B^*A\leq 0$, then
\[|A+B|\leq |A|+|B|.\]
\end{thm}

\begin{proof}
Clearly,
\[(A+B)^*(A+B)=A^*A+A^*B+B^*A+B^*B.\]
As $A^*B+B^*A\leq 0$, then
\[A^*A+A^*B+B^*A+B^*B\leq A^*A+B^*B.\]
By Theorem \ref{Lowner-Heinz inequality theorem}, we have
\[|A+B|=\sqrt{A^*A+A^*B+B^*A+B^*B} \leq \sqrt{A^*A+B^*B}.\]
Since $AB=BA$ and $A$ is normal, Proposition \ref{commutativity
operators equiv abs va pro} implies that $|A||B|=|B||A|$ or
$|A|^2|B|^2=|B|^2|A|^2$. Finally, Lemma \ref{sqrt A+B Lemma} does
the remaining job, i.e. it gives us
\[|A+B|\leq |A|+|B|\]
and this completes the proof.
\end{proof}

\section*{acknowledgement} The author wishes to thank Mr. S. Dehimi
for a discussion which led to a slight improvement of the result of
Theorem \ref{triangle inequality abs valu THM}.

\end{document}